\documentclass[12pt,a4paper,twoside]{amsart}
\usepackage{amsmath}
\usepackage{amssymb}
\usepackage{amsfonts}
\usepackage{enumitem}

\newtheorem{theorem}{Theorem}[section]
\newtheorem{lemma}[theorem]{Lemma}

\newtheorem{corollary}[theorem]{Corollary}
\theoremstyle{definition}

\newcommand{\GG}{\Gamma}

\newcommand{\kk}{\kappa}

\newcommand{\CC}{\mathbb{C}}

\newcommand{\ZZ}{\mathbb{Z}}
\newcommand{\FF}{\mathbb{F}}

\newcommand{\PP}{\mathbb{P}}

\newcommand{\calF}{\mathcal F}
\newcommand\calZ{\mathcal Z}
\newcommand\Del{\Delta}

\newcommand{\nc}{\newcommand}
\nc{\on}{\operatorname}

\nc{\ol}{\overline}

\nc{\Bun}{\on{Bun}}

\nc{\mc}{\mathcal}

\newcommand{\OO}{{\mathbb{O}}}

\newcommand{\calL}{\mathcal L}

\newcommand{\GL}{\operatorname{GL}}
\newcommand{\Pic}{\operatorname{Pic}}
\renewcommand{\Bun}{\operatorname{Bun}}

\newcommand{\calO}{\mathcal O}

\newcommand{\Lam}{\Lambda}
\newcommand{\lam}{\lambda}

\newcommand{\calD}{\mathcal D}

\newcommand{\PGL}{\operatorname{PGL}}
\newcommand{\kap}{\kappa}

\newcommand{\calE}{\mathcal E}

\newcommand{\Aut}{\operatorname{Aut}}
\newcommand{\qlb}{\overline{\mathbb Q}_{\ell}}
\newcommand{\Fr}{\operatorname{Fr}}

\newcommand{\Ome}{\Omega}

\newcommand{\calH}{\mathcal H}

\newcommand\gam{\gamma}

\newcommand{\Spec}{\operatorname{Spec}}

\newcommand{\SL}{\operatorname{SL}}
\newcommand{\bpo}{\Bun_G(\PP^1)_{0,\infty}}
\renewcommand\kk{\mathbf k}
\newcommand\KK{\mathbb K}

\newcommand\gr{\operatorname{gr}}
\newcommand\okk{\overline\kk}
\newcommand\x{\times}
\newcommand\End{\operatorname{End}}
\newcommand\cusp{\operatorname{cusp}}
\newcommand\Irr{\operatorname{Irr}}
\newcommand\Gal{\operatorname{Gal}}
\newcommand\act{\mathbf{act}}

\newcommand\calC{\mathcal C}
\newcommand\Lift{\operatorname{Lift}}
\newcommand\calR{\mathcal R}
\newcommand\Gam{\Gamma}
\newcommand\ind{\operatorname{ind}}
\newcommand\Norm{\operatorname{Norm}}
\usepackage{xcolor}
\usepackage{amscd}

\usepackage{amscd}

\begin{document}
{\color{red}}
\title[Hecke algebras and bundle on $\PP^1$]{Hecke algebras for the 1st congruence subgroup and bundles on $\PP^1$ I: the case of finite field}
\author{Alexander Braverman and David Kazhdan}
\begin{abstract}
Let  $G$ be  a split reductive group over a finite field $\kk$
In this note
we study the space $V$ of finitely supported functions on the set $\bpo(\kk)$ of isomorphism classes $G$-bundles on the projective line $\PP^1$ endowed with a trivialization at $0$ and $\infty$. We show that $V$ is naturally isomorphic to the regular bimodule over the Hecke algebra $A$ of the group $G(\kk((t)))$ with respect to the first congruence subgroup. As a byproduct we show that Hecke operators at points different from $0$ and $\infty$ to generate   the "stable center" of $A$. We provide an expression of the character of the lifting of an irreducible cuspidal representation of $\GL(N,\kk)$ to $\GL(N,\kk')$ where $\kk'$ is a finite extension of $\kk$ in terms of these generators.

In a subsequent publication we plan to develop analogous constructions in the case when $\kk$ is replaced by a local non-archimedian field.
\end{abstract}
\maketitle
\section{Introduction and setup}

\subsection{}\label{setup} Let $\kk$ be the finite field $\FF_q$ with a fixed algebraic closure $\okk$.
For every $i\geq 1$ let $\kk_i$ denote the degree $i$ extension of $\kk$ inside $\okk$; we have the norm
map $\Norm_{i,j}:\kk_i^{\x}\to \kk_j^{\x}$ for every $i\geq j$. In particular the groups $\kk_i^{\x}$ form a projective system.
Let also $\KK=\kk((t)), \OO=\kk[[t]]$.

Let $G$ be a split connected reductive group over $\kk$.

Denote by $K_1$ the first congruence subgroup of $G(\OO)$ (consisting of elements equal to 1 modulo $t$).
We denote by $A$ the Hecke algebra of $G(\KK)$ with respect to $K_1$ with coefficients in an algebraically closed field $F$ of characteristic $0$; this is an associative algebra containing the group algebra of $G(\kk)$ as a subalgebra.
We denote by
 $G^{\vee}$ the Langlands dual group of $G$; unless otherwise stated we shall think about it as a group over $F$.

Let $\calZ(A)$ denote the center of $A$. It follows from \cite{GeLa} that $\calZ(A)$ contains the algebra $F[Z(G^{\vee},\kk)]^{G^{\vee}}$ of $G^{\vee}$-invariant functions on the scheme $Z_{G^{\vee},\kk}$ consisting of pairs $(s,u)$
where

1) $s\in G^{\vee}$

2) $u: \underset{\leftarrow}\lim \ \kk_i^{\x}\to G^{\vee}$

\noindent
subject to the relation $sus^{-1}=u^q$ (it is easy to see that $Z_{G^{\vee},\kk}$ is indeed a scheme of finite type over $F$; this follows from the fact that there exists some  $i$  depending only on $G$ such that for any pair $(s,u)$ as above the homomorphism $u$ factorizes through $\kk_i^{\x}$).

One of the purposes of this note is to give an alternative characterization of the homomorphism $F[Z(G^{\vee},\kk)]^{G^{\vee}}\to \calZ(A)$ (in the future we shall call the image of this homomorphism {\em the stable center of $A$} and we shall denote it by $\calZ_{st}(A)$).
For this we are going to use some global constructions discussed in the next subsection.
\subsection{Bundles on $\PP^1$ and $\calZ(A)$}\label{bundles}
Let $\bpo$ denote the moduli stack of bundles on $\PP^1$ (over $\kk$) with trivializations at $0$ and $\infty$.
Let $V$ denote the space of $F$-valued functions on $\bpo(\kk)$ with finite support.
The space $V$ has two natural commuting actions of the algebra $A$ (by Hecke operators at $0$ and $\infty$).
Our first statement is that $V$ is essentially the regular bimodule over $A$. More precisely, we claim that there exists an anti-involution
$\iota:A\to A$ on $A$ such that if we use it
to turn $V$ into an $A$-bimodule -- keeping the (left) action of $A$ at $0$ intact and making the action of $A$ at $\infty$ into a right action by using $\iota$, then we have
\begin{theorem}\label{loc-glob}
There is a natural isomorphism $V\simeq A$ of $A$-bimodules.
\end{theorem}
\begin{corollary}\label{loc-glob-cent}
We have $\calZ(A)=\End_{A\otimes A}(V)$.
\end{corollary}
We do not know any simple description of $\iota$ but it will satisfy the following properties:

1) It preserves the subalgebra $F[G(\kk)]\subset A$, and on this subalgebra it is induced by the map $g\mapsto g^{-1}$.

2) Let $e_{fin}$ be the idempotent $\frac{1}{\# G(\kk)}\sum_{g\in G(\kk)} g \in F[\GG(\FF_q)]\subset A$. Then the induced anti-involution on $e_{fin}A e_{fin}$ is equal to identity (note that $e_{fin}A e_{fin}$ is the spherical Hecke algebra of $G$ which is commutative).

\noindent
Let us also note that $\iota$ is not completely local in the sense that it depends on the choice of a generator of the maximal ideal of $\kk[[t]]$.

We shall explain the strategy of the proof of Theorem \ref{loc-glob} in Subsection \ref{sketch}.
\subsection{Hecke operators at other points and stable center}
Let $D$ be a subset of $\okk^{\x}$ invariant under the Galois group of $\okk$ over $\kk$. Let also $f\in F[G^{\vee}]^{G^{\vee}}$.
To this data one can in the standard way associate the Hecke operator $h_{D,f}\in \End_{A\otimes A}(V)$. According to
Corollary \ref{loc-glob-cent} we can regard $h_{D,f}$ as an element of $\calZ(A)$.

\begin{theorem}\label{cent-div}
\begin{enumerate}[label=(\roman*)]
\item
We have $h_{D,f}\in \calZ_{st}$. In particular, $h_{D,f}$ corresponds to a function $\phi_{D,f}$ on $Z(G^{\vee},\kk)$.
\item
Assume that $D$ is the $\Gal(\okk/\kk)$-orbit of some $x\in \kk_i^{\x}$ which does not belong to any $\kk_j$ with $j<i$.
Then
\begin{equation}\label{formula}
\phi_{D,f}(s,u)=f(s^i\, u(y))
\end{equation}
where $y\in \kk_n^{\x}$  for some $n\geq i$ is such that $u$ factorizes through $\kk_n^{\x}$, $n$ is divisible by $i$ and $\Norm_{n,i}(y)=x$ (in particular, we claim that the RHS of (\ref{formula}) does not depend on the choice of $n$ and $y$).
\end{enumerate}
\end{theorem}
It is easy to see that Theorem \ref{cent-div} implies the following
\begin{corollary}
$\calZ_{st}(A)$ is the subalgebra of $\calZ(A)$ generated by all the $h_{D,f}$.
\end{corollary}
\subsection{Action on cuspidal representations of $G(\kk)$}\label{cusp-int}
It is clear that stack $\bpo$ is acted on by $G\times G$ (by changing the trivializations at $0$ and $\infty$), hence the group $G(\kk)\times G(\kk)$ acts on $V$.
Let $V_{\cusp}$ denote the maximal cuspidal $G(\kk)\times G(\kk)$-subrepresentation  of $V$ (it is easy to see that
$V_{\cusp}$ is also the maximal cuspidal $G(\kk)$-subrepresentation  of $V$ with respect to either action of $G(\kk)$ -- this follows immediately from Theorem \ref{loc-glob}). It is clear that $V_{\cusp}$ is invariant under the action of $A\otimes A$ and also under the operators $h_{D,f}$.
s

The following lemma is easy:
\begin{lemma}\label{lem-cusp}
Assume that $G$ is semi-simple. Then
$V_{\cusp}=F_{\cusp}(G(\kk))$ (i.e. the cuspidal part of the regular representation of $G(\kk)$.
\end{lemma}
In particular, the restriction of every $h_{D,f}$ to $V_{\cusp}$ defines a $G(\kk)\times G(\kk)$-endomorphism of $F_{\cusp}(G(\kk))$. Such a data is equivalent to a function $\eta_{D,f}:\Irr_{\cusp}(G(\kk))\to F$ where
$\Irr_{\cusp}(G(\kk))$ denotes the set of isomorphism classes of irreducible cuspidal representations of $G(\kk)$.
We {\color{red} will now} describe this function explicitly in the case when $G=\PGL(N)$ and $f$ is the trace function.

First of all,
for every $a\geq 1$ let us denote by $\calR_a$ the Grothedieck group of representations of $\PGL(N,\kk_a)$;
we shall write $\calR$ instead of $\calR_1$. We have a natural map $\Lift_a:\calR\to \calR_a$ defined as follows.
It is known (that $\calR_a$ is freely spanned by the classes of Deligne-Lusztig virtual representations $R(T,\theta)$ where
$T$ is a maximal torus in $\PGL(N)$ defined over $\kk$ and and $\theta:T(\kk)\to \CC^{\x}$ (the representation $R(T,\theta)$ depends only on the pair $(T,\theta)$ up to conjugacy). Given $a\geq 1$ and a pair $(T,\theta)$ we get by extension of scalars a torus $T_a$ in $\PGL(N,\kk_a)$ endowed with the norm map $T_a(\kk)\to T(\kk)$; composing this map with $\theta$ we get a character $\theta_a$ of $T_a(\kk)$. The map $\Lift_a$ is characterized by the condition that it sends the class of $R(T,\theta)$ to the class of $R(T_a,\theta_a)$.

We shall also assume that $D$ is as in Theorem \ref{cent-div}(2).
Then we have the following.
\begin{theorem}\label{lift-gln}
Under the above conditions on $D$ and $f$ we have:
\begin{enumerate}[label=(\roman*)]
\item
$\eta_{D,f}=0$ if $i$ is not divisible by $N$.
\item
Assume that $i=N$. Any $x\in \kk_i^{\x}=\kk_N^{\x}$ can be thought of as an element of $\GL(N,\kk)$ defined up to conjugacy; abusing the notation we denote its image in $\PGL(N,\kk)$ by the same letter.
Then for any $\pi\in\Irr_{\cusp}(\PGL(N,\kk))$ we have
\begin{equation}
\eta_{D,f}(\pi)=\chi_{\pi}(x).
\end{equation}
Here $\chi_{\pi}$ is the character of $\pi$.
\item
Let now $i=aN$ with $a\geq 1$. Then
\begin{equation}\label{lifting}
\eta_{D,f}(\pi)=\chi_{\Lift_a(\pi)}(x)
\end{equation}
\end{enumerate}
\end{theorem}

Note that (\ref{lifting}) provides an explicit description of the character of $\Lift_a(\pi)$ on the set of elliptic elements in $\PGL(N,\kk_a)$.

\subsection{Sketch of the proof of Theorem \ref{loc-glob}}\label{sketch}
The proof of Theorem \ref{loc-glob} will be divided into the following steps (we sketch the proof here, since we think that
it is of independent interest).

{\em Step 1.} Since the algebra $A$
naturally acts on $V$, in order to define a map $\act:A\to V$ it is enough to choose  an element $\delta\in V$ and set $\act(a)=a\underset{0}\star \delta$ where $\underset{0}\star$ stands for the action of $A$ at $0$. We then need to show that $\act$ is an isomorphism -- i.e. that $V$ is a free module over $A$ of rank 1 with generator $\delta$. We shall choose $\delta$ to be the characteristic function of the trivial $G$-bundle on $\PP^1$ with the tautological trivialization at $0$ and $\infty$.

On the other hand $V$ is endowed with a natural involution coming from the map $z\to z^{-1}$ on $\PP^1$ which fixes $\delta$. If we know that $\act$ is an isomorphism, then we can transport this involution to a map $\iota:A\to A$.
Since $\delta$ is preserved by the above involution of $V$ we obviously have
$$
a\underset{0}\star \delta=\iota(a)\underset{\infty}\star \delta
$$
which immediately implies that $\iota(ab)=\iota(b)\iota(a)$, thus $\iota$ is an anti-involution of $A$. The verification of properties 1) and 2) after Corollary \ref{loc-glob-cent} is straightforward and we leave it to the reader.

{\em Step 2.} We now need to show that $\iota$ is an isomorphism.
Let $\Lam$ denote the coweight lattice of $G$ and let $\Lam^+$ be the set of dominant coweights. Both sets
are endowed with a standard partial order ($\lam>\mu$ if $\lam-\mu$ is a sum of positive roots), and both $A$ and $V$ acquire a natural filtration numbered by elements of $\Lam^+$ and the map $\act$ is compatible with the filtrations.
We shall denote the corresponding filtration components by $A_{\leq \lam}$ and $V_{\leq \lam}$.

These filtrations are
defined as follows. First, the double quotient $K_1\backslash G(\KK)/K_1$ maps naturally to $G(\OO)\backslash G(\KK)/G(\OO)$. It is known that the latter quotient is naturally identified with $\Lam^+$; we shall denote by $(G(\OO)\backslash G(\KK)/G(\OO))_{\leq \lam}$ the union of all double cosets corresponding to $\mu\leq \lam$. Similarly, we define $(K_1\backslash G(\KK)/K_1)_{\lam}$ to be the preimage of $(G(\OO)\backslash G(\KK)/G(\OO))_{\leq \lam}$ in $(G(\OO)\backslash G(\KK)/G(\OO))_{\leq \lam}$.
We then define $A_{\leq \lam}$ to be the subspace of $A$ consisting of functions supported on $(G(\OO)\backslash G(\KK)/G(\OO))_{\leq \lam}$. This is an algebra filtration on $A$.

Similarly, the set of isomorphism classes of $G$-bundles on $\PP^1$ is naturally identified with $\Lam^+$; hence the stack
$\bpo$ is stratified by elements of $\Lam^+$ and we define $V_{\leq \lam}$ similarly to $A_{\leq \lam}$.

{\em Step 3.} Since the map $\act$ is compatible with the above filtrations, in order to prove that $\act$ is an isomorphism, it is enough to show that it induces an isomorphism on the accociatd  graded spaces. So, we set $A_{\lam}=A_{\leq \lam}/A_{<\lam}$ and
$V_{\lam}=V_{\leq\lam}/V_{< \lam}$. We also denote by $\act_{\lam}$ the corresponding map $A_{\lam}\to V_{\lam}$.

To each $\lam$ we can naturally associate a parabolic subgroup $P_{\lam}$ of $G$ with a Levi subgroup $M_{\lam}$.
We also denote by $P_{\lam}^-$ the opposite parabolic (so $P_{\lam}\cap P_{\lam}^-=M_{\lam}$.
The we prove the following
\begin{theorem}\label{main-radon}
\begin{enumerate}[label=(\roman*)]
\item
We have natural isomorphisms
\begin{equation}\label{formula-global}
V_{\lam}\simeq F(((G/U_{\lam}\underset{M_{\lam}}\times G/U_{\lam}))(\kk))
\end{equation}
and
\begin{equation}\label{formula-local}
A_{\lam}\simeq F(((G/U_{\lam}\underset{M_{\lam}}\times G/U_{\lam}^-))(\kk)).
\end{equation}
Here $\underset{M_{\lam}}\times$ stands for the quotient of the product by the diagonal action on $M_{\lam}$.
\item
Under the above isomorphisms the map $\act_{\lam}$ is the standard intertwining operator corresponding to the pair $(P_{\lam},P_{\lam}^-)$ acting on the 2nd factor (cf. Section \ref{inter} for precise definitions).
\end{enumerate}
\end{theorem}

{\em Step 4.} According to Theorem \ref{main-radon} it is enough to show that given two associate parabolic subgroups $P$ and $Q$ of $G$
\footnote{I.e. parabolic subgroups which have a common Levi subgroup}
 with unipotent radicals $U_P,U_Q$ the standard intertwining operator $I_{P,Q}:F(G/U_P(\kk))\to F(G/U_Q(\kk))$ is an isomorphism. This is proven in \cite{DD} (cf. also Theorem 2.4 in \cite{HoLe}).

\subsection{Acknowledgements} We thank R.~Bezrukavnikov and M.~Finkelberg for very useful discussions.

\section{Bundles on $\PP^1$ with trivializations and proof of (\ref{formula-global})}

\subsection{The stack of $G$-bundles on $\PP^1$}
We denote by $\Bun_G(\PP^1)$ the moduli stack of $G$-bundles on $\PP^1$. It is well-known that
$\Bun_G(\PP^1)(\kk)$ naturally identifies with $\Lam^+$. This identification can be described as follows.
Let $\lam\in \Lam$; this is homomorphism from $\GG_m$ to the (chosen) maximal torus $T$ in $G$; thus $\lam$ defines a homomorphism $\GG_m\to G$. Applying it to the line
bundle $\calO(1)$ on $\PP^1$ we get a $G$-bundle $\calO(\lam)$ whose isomorphism class only depends on the orbit of $\lam$ under the Weyl group $W$, and the corresponding map $\Lam^+\to \Bun_G(\PP^1)(\kk)$ is a bijection.
Note that for any $z\in \PP^1(k)$ we can write canonically $\calO(1)=\calO(z)$. Thus the bundle $\calO(\lam)$ acquires a canonical trivialization away from $z$. In what follows we shall choose such a $z$ so that $z\in \kk^{\x}$.
Some of our construction will slightly depend on this choice (in principle, one can choose $z=1$ but there is nothing special about this choice).

Geometrically this defines a stratification of $\Bun_G(\PP^1)$ -- we denote by $\Bun_G(\PP^1)^{\lam}$ the corresponding locally closed
substack of $\Bun_G(\PP^1)$. Denote by $\overline{\Bun_G(\PP^1)^{\lam}}$ its closure. Then it is easy to see that
$\Bun_G(\PP^1)^{\mu}\subset \overline{\Bun_G(\PP^1)^{\lam}}$ if and only if $\lam\leq \mu$.
We set
$$
\Bun_G(\PP^1)^{\leq \lam}=\bigcup\limits_{\mu\leq \lam} \Bun_G(\PP^1)^{\mu}.
$$
This is an open substack of $\Bun_G(\PP^1)$.
\subsection{Parabolic subgroups and Harder-Narasimhan filtrations}
Recall
that we have chosen a Cartan subgroup $T$ of $G$ and a Borel subgroup $B$ containing it.
Every $\lam\in \Lam$ defines a parabolic subgroup $P_{\lam}$ in $G$ (its unipotent radical is generated by root vectors corresponding to roots $\alpha$ such that $\langle \lam,\alpha\rangle >0$). We have also the opposite parabolic $P_{\lam^-}$ which have a common Levi subgroup $M_{\lam}$. We denote by
$\pi_{\lam}:P_{\lam}\to M_{\lam}$, $i_{\lam}:M_{\lam}\to P_{\lam}$ the natural maps (similarly for $\pi_{\lam}^-$ and $i_{\lam}^-$).

By construction every $\calO(\lam)$ is endowed with a $T$-structure. However, this structure is not canonical in the sense that it is not invariant under the group of automorphisms of $\calO(\lam)$ -- in particular, it is not  well-defined on a $G$-bundle abstractly isomorphic to $\calO(\lam)$. However, it is easy to see that the induced $P_{\lam}$-structure
is invariant under the automorphisms of $\calO(\lam)$. For $G=\GL(N)$ this is exactly the Harder-Narasimhan filtration, so for general $G$ we shall also sometimes use this term for the $P_{\lam}$ structure on $G$-bundle isomorphic to $\calO(\lam)$.

\subsection{The main stack}
Our main geometric object in this Section is the stack $\bpo$ classifying $G$-bundles on $\PP^1$ with trivializations at $0$ and $\infty$.
We denote by $\bpo^{\leq \lam}$ and $\bpo^{\lam}$ the corresponding (open and locally closed) substacks of $\bpo$ which are preimages respectively of $\Bun_G(\PP^1)^{\leq \lam}$ and of $\Bun_G(\PP^1)^{\lam}$. As before we set $V$ to be the space of $F$-valued functions on $\bpo(\kk)$ with finite support.
We also set $V_{\leq \lam}$ (resp. $V_{\lam}$) to be the space of $F$-valued functions on $\bpo^{\leq \lam}(\kk)$ (resp. on
$\bpo^{\lam}(\kk)$). Then $V_{\leq \lam}$ define a filtration on $V$ with $V_{\lam}$'s being the corresponding associated graded quotients.

\begin{theorem}\label{global}
There exists a map $\eta_{\lam}:\bpo^{\lam}\to G/U_{\lam}\underset{M_{\lam}}\times G/U_{\lam}$ such that:
\begin{enumerate}[label=(\roman*)]
\item
$\eta_{\lam}$ is surjective on $\kk$-points
\item The fibers of $\eta_{\lam}$ are classifying stacks of some unipotent group.
\end{enumerate}
In particular (i) and (ii) imply that $\eta_{\lam}$ is an isomorphism on the level of $\kk$-points.
\end{theorem}
\begin{corollary}
$V_{\lam}\simeq F(G/U_{\lam}\underset{M_{\lam}}\times G/U_{\lam}(\kk))$ (this is the statement of (\ref{formula-global})).
\end{corollary}
\begin{proof}
Let us begin by observing that the variety $G/U_{\lam}\underset{M_{\lam}}\times G/U_{\lam}$ is the moduli space of the following data: two $P_{\lam}$-structures on the trivial $G$-torsor together with an isomorphism between the corresponding $M_{\lam}$-torsors.

Now any bundle on $\PP^1$ with isomorphism type $\lam$ is endowed with a canonical $P_{\lam}$-structure (for $G=GL(N)$ this is the Harder-Narasimhan filtration). By looking at the image of this structure at $0$ and $\infty$ we get a map
from $\bpo^{\lam}$ to $G/P_{\lam}\times G/P_{\lam}$. In addition, let us recall that we have chosen some $z\in \PP^1(\kk)\backslash \{ 0,\infty\}$. Then we claim that the corresponding $M_{\lam}$-torsors are canonically isomorphic. Indeed, it is enough to check it for $G=\GL(N)$. In this case, it follows from the fact the every graded piece of the associated graded of any rank $N$ vector bundle with respect to its Harder-Narasimhan filtration is
isomorphic to $\calL^k$ for some $\calL\in \Pic(\PP^1)$, $k\in \ZZ_{>0}$. Now we have unique up to constant isomorphism
$\calL\simeq \calO(d\cdot z)$ where $d=\deg(\calL)$, which identifies the fibers of $\calL$ at $0$ and $\infty$.
Thus we get a map $\eta_{\lam}:\bpo^{\lam}\to G/U_{\lam}\underset{M_{\lam}}\times G/U_{\lam}$.

We now need to prove (1) and (2). To prove (1) let us observe that $\bpo^{\lam}$ is acted on transitively by the group $G\times G$, and the same is true for $G/U_{\lam}\underset{M_{\lam}}\times G/U_{\lam}(\kk)$; moreover, this action is transitive on the latter.

To check (2) let us describe the stack $\bpo^{\lam}$ explicitly. Let $\calF$ be a $G$-bundle of type $\lam$ with trivializations at $0$ and $\infty$. Let $\Aut(\calF)$ be the group of automorphisms of $\calF$ (just as a bundle -- i.e.
the automorphisms are not supposed to preserve the trivializations).
Looking at the action of automorphisms at $0$ and $\infty$ we get a map $\Aut(\calF)\to G\times G$.
The stack $\bpo^{\lambda}$ is equivalent to $G\times G/\Aut(\calF)$.\footnote{The map $\Aut(\calF)\to G\times G$ is in general not injective, thus the quotient is indeed a stack.}
Let us determine the image of $\Aut(\calF)$ in $G\times G$. For this it is enough to take any $\calF$ as above, and we take
$\calF=\calO(\lambda)$ with trivializations at $0$ and $\infty$ coming from the trivialization of $\calO(\lambda)$ on $\PP^1\backslash\{z\}$. Then the image of $\Aut(\calF)$ in $G\times G$ is easily seen to be the subgroup of $P_{\lam}\times P_{\lam}$ consisting of pairs $(p,p')$ whose images in $M_{\lam}$ coincide. Thus the quotient of $G\times G$ by the image of $\Aut(\calF)$ is precisely $G/U_{\lam}\underset{M_{\lam}}\times G/U_{\lam}$. It remains to note that the kernel of the map
$\Aut(\calF)\to G\times G$ is unipotent.
\end{proof}

\section{Hecke stack for the 1st congruence subgroup}

\subsection{The Hecke stack in the spherical case}
Let $\calH_{\OO}$ be the stack classifying triples $(\calF _1,\calF_2,\kap)$

where

1) $\calF_i$ is a $G$-bundle on the formal disc $\calD=\Spec(\OO)$.

2)  $\kap$ is an isomorphism $\calF_1|_{\calD^*}\to \calF_2|_{\calD^*}$ where $\calD^*=\Spec(\KK)$.

This stack can be identified with the double quotient $G(\OO)\backslash G(\KK)/G(\OO)$.
As before we have the natural bijection $\calH_{\OO}(\kk)=\Lam^+$. Here every $\lam\in\Lam^+$ corresponds
to the element $t^{\lam}=\lam(t)\in T(\KK)\subset G (\KK)$. In fact, the stack $\calH_{\OO}$ endowed with a natural stratification
$$
\calH_{\OO}=\bigsqcup\limits_{\lam\in\Lam^+}\calH^{\lam},
$$
where $\calH_{\OO}^{\lam}$ is the (reduced) locally closed substack of $\calH_{\OO}$ whose only $\kk$-point is given by the above construction. The closure $\calH^{\leq \lam}_{\OO}$ of $\calH_{\OO}^{\lam}$ is equal to the union of all $\calH_{\OO}^{\mu}$ with $\mu\leq \lam$.
\subsection{The Hecke stack for the 1st congruence subgroup}
Let $\calH$ denote the stack of 5-tuples $(\calF_1,\calF_2, \phi_1,\phi_2,\kap)$ where

1) $\calF_i$ is a $G$-bundle on the formal disc $\calD=\Spec(\OO)$ where $\OO=\kk[[t]]$.

2) $\phi_i$ is the trivialization of the fiber of $\calF_i$ at $0$.

3) $\kap$ is an isomorphism $\calF_1|_{\calD^*}\to \calF_2|_{\calD^*}$ where $\calD^*=\Spec(\KK)$ where $\KK=\kk((t))$.

\noindent
Let $\calH^{\lam}$ denote the preimage of $\calH^{\lam}_{\OO}$ under the natural map $\calH\to \calH_{\OO}$; this is a locally closed substack of $\calH$. Similarly, we denote by $\calH^{\leq \lam}$ the preimage of $\calH^{\leq\lam}_{\OO}$.

\subsection{The algebra $A$ and its filtration}
Let $A$ denote the space of finitely supported $F$-valued functions on $\calH(\kk)$. This is an algebra with respect to convolution.
We let $A_{\leq \lam}$ be the space of functions on $\calH^{\leq \lam}$. It is easy to see that $A_{\lam}\cdot A_{\mu}\subset A_{\leq \lam+\mu}$; thus $A_{\leq \lam}$ is a filtration on $A$ by dominant coweights.
The associated graded algebra is the direct sum of all the $A_{\lam}$ where $A_{\lam}$ is the space of $F$-valued functions
on $\calH^{\lam}$.

We would like to describe the space $A_{\lam}$. For this we are going to need the following geometric description of the stack $\calH^{\lam}$.

\begin{theorem}\label{local}
There is a map $\calH^{\lam}\to G/U_{\lam}\underset{M_{\lam}}\times G/U_{\lam}^-$ which is surjective on $\kk$-points and whose fibers are classifying spaces of some pro-unipotent groups (in particular, this map is bijective of $\kk$-points).
\end{theorem}

To construct this map we need a digression on Jantzen filtrations.
\subsection{Jantzen filtrations}
Let $\calE,\calE'$ be two finitely generated torsion-free $\OO$-modules endowed with an isomorphism
$\kap:\calE\underset{\OO}\otimes \KK \to \calE'\underset{\OO}\otimes \KK$.
Let $E,E'$ be the fibers of $\calE,\calE'$ at $0$.

Set $E_i$ be the set of all $v\in E$ such that for some (equivalently for any) $s\in \calE$ with $s(0)=v$ we have $t^{-i}\kap(s)\in \calE'$. Similarly we define $E'_i$ by replacing $\kap$ by $\kap^{-1}$.
Then it is clear that

(1) $E_i$ is an increasing filtration on $E$, $E_i'$ is an increasing filtration on $E'$ (we are going to call them
Jantzen filtrations).

(2) $\kap$ establishes a canonical isomorphism between $\gr_i(E)$ and $\gr_{-i}(E')$.

\noindent
Let $N=\dim(E)=\dim(E')$. Then we can view $\calE$ and $\calE'$ as rank $N$ vector bundles on
$\calD$, and $\kap$ defines a point of the double quotient stack $\GL(N,\OO)\backslash \GL(N,\KK)/\GL(N,\OO)$.
On the other hand, the above Jantzen fitrations define a reduction of $V$ and $V'$ to some parabolic subgroups
$P,P'$ of $\GL(N)$. Moreover, (2) above implies that the type of $P'$ is opposite to that of $P$.

Recall that $\kk$-points of $\GL(N,\OO)\backslash \GL(N,\KK)/\GL(N,\OO)$ are in one-to-one correspondence with $\lam\in\Lam^+$ (where $\Lam^+$ denotes the set of dominant coweights of $\GL(N)$).
\begin{lemma}
Assume that $(\calE,\calE',\kap)$ has type $\lam$. Then the Jantzen filtration on $E$ defines a reduction of $E$ to $P_{\lam}$. Similarly, the Jantzen filtration on $E'$ defines a reduction of $E'$ to $P_{\lam}^-$.
\end{lemma}
As before, to prove this lemma, it is enough to look at the case $\calE=\calE'=\OO^N$ and $\kap$ being the diagonal matrix with entries $(t^{n_1},\cdots,t^{n_N})$ where $\lam=(n_1,\cdots,n_N)$, in which case the above lemma is obvious.

Since Jantzen filtrations are completely canonical, they are functorial and are compatible with tensor products.
As a corollary, we get the following
\begin{lemma}\label{lem-loc}
Let $G$ be as before, and let $(\calF_1,\calF_2,\kap)$ be a $\kk$-point of the stack
$G(\OO)\backslash G(\KK)/G(\OO)$ of type $\lam$. Let $F_i$ be the fiber of $\calF_i$ at $0$. Then $F_1$ is endowed with a
$P_{\lam}$-structure, $F_2$ is endowed with a $P_{\lam}^-$-structure and the corresponding $M_{\lam}$-bundles are canonically isomorphic.
\end{lemma}
\subsection{End of the proof of Theorem \ref{local}}
It is clear that $G/U_{\lam}\underset{M_{\lam}}\times G/U_{\lam}^-$ is the moduli space of the following data:

1) A reduction of the trivial $G$-torsor to $P_{\lam}$.

2) A reduction of the trivial $G$-torsor to $P_{\lam}^-$.

3) An isomorphism between the corresponding $M_{\lam}$-torsors.

\noindent
Thus using Lemma \ref{lem-loc} we immediately get a map $\calH^{\lam}\to G/U_{\lam}\underset{M_{\lam}}\times G/U_{\lam}^-$.
We now need to prove that it is surjective on $\kk$-points and that the fibers are classifying stacks of some pro-unipotent group. The former follows again from the fact that $G(\kk)\times G(\kk)$ acts transitively on $G/U_{\lam}\underset{M_{\lam}}\times G/U_{\lam}^-(\kk)$. To prove the latter as before we note that $\calH^{\lam}$ is equivalent to
$G\times G/A_{\lam}$ where $A_{\lam}$ is the group of automorphisms of the triple $(\calF_1,\calF_2,\kap)$ where
$\calF_i$ are trivial bundles and $\kap=t^{\lam}$ (it maps to $G\times G$ by looking at its actions on the fibers of $\calF_i$ at $0$). In other words, $A_{\lam}$ is the subgroup of $G(\OO)$ consisting of elements $g\in G(\OO)$ such that
$\kap g \kap^{-1}\in G(\OO)$ and the map to $G\times G$ is given by $(g\mod t, \kap g \kap^{-1}\mod t)$.
It is easy to see that the kernel of the map $A_{\lam}\to G\times G$ is pro-unipotent and
the image of $A_{\lam}$ in $G\times G$ is equal to
$$
\{(p,p')\in P_{\lam}\times P_{\lam}^-|\ \pi_{\lam}(p)=\pi_{\lam}^-(p')\}.
$$
which finishes the proof.
\section{Hecke correspondences and proof of Theorem \ref{main-radon}}\label{inter}
The purpose of this Section is to prove Theorem \ref{main-radon}. As was explained in the Introduction this implies Theorem
\ref{loc-glob}.
\subsection{Intertwining operators}
We begin by recalling the definition and properties of the standrd intertwining operators for induced representations of the group $G(\kk)$.

Let $P$ and $P^-$ be any two opposite parabolics in $G$ -- i.e. two parabolics whose intersection $M$ is a Levi subgroup in both. Let $\gamma$ (resp. $\gamma_-$) be the natural map $G\to G/U$ (resp. $G\to G/U_-$).
Then we can view $G$ as a correspondence between $G/U_-$ and $G/U$ and we define
$\Phi_{P_-,P}: F(G/U^-(\kk))\to F(G/U(\kk))$ be the corresponding intertwining operator (a.k.a. "Radon transform").
This operators commutes with the natural action of $G(\kk)\times M(\kk)$ on both sides.
The following result is proven in \cite{DD} (cf. also \cite{HoLe}):
\begin{theorem}
The operator $\Phi_{P_-,P}$ is an isomorphism.
\end{theorem}

\subsection{The map $\act$ and the stack $\calC$}
The algebra $A\otimes A$ acts naturally on the space $V$ by convolution at $0$ and $\infty$; as in subsection \ref{sketch} we denote these actions by $a\underset{0}\star v$ and $a\underset{\infty}\star v$.

Let $\delta\in V$ be the characteristic function of the $\kk$-point of $\bpo$ corresponding to the trivial bundle with tautological trivializations at $0$ and $\infty$. We define the map $\act:A\to V$ by setting
$$
\act(a)=a\underset{0}\star \delta.
$$
It is easy to see that $\act(a)$ is also equal to $\iota(a)\underset{\infty}(\delta)$ where $\iota$ is defined in subsection
\ref{bundles}. Indeed

\subsection{The space $\calC^{\lam}$}\label{zlam}
Let $\calC^{\lam}$ denote the moduli space of the following data:

(1) A $G$-bundle on $\calF$ on $\PP^1$ of isomorphism class $\lam$.

(2) Trivialization of the fiber of $\calF$ at $0$.

(3) A trivialization of $\calF|_{\PP^1\backslash\{ 0\}}$ which has type $\lam$ near $0$.

We have obvious maps
$$
\begin{CD}
\calC^{\lam} @>>> \bpo^{\lam} \\
@VVV  \\
\calH^{\lam}
\end{CD}
$$
Here the horizontal arrow is obtained by looking at (1),(2) and (3) above and using only the image of (3) at $\infty$, and the vertical arrow is obtained by restricting all of the data to the formal disc near $0$ (and by observing that the trivial bundle is endowed with a canonical trivialization of its fiber at $0$).
Thus in view of the previous sections we get natural maps
$$
\begin{CD}
\calC^{\lam} @>{\beta}>> G/U_{\lam}\underset{M_{\lam}}\times G/U_{\lam} \\
@V{\alpha}VV  \\
G/U_{\lam}\underset{M_{\lam}}\times G/U^-_{\lam}.
\end{CD}
$$
Note that
$$G/U_{\lam}\underset{M_{\lam}}\times G/U_{\lam}=G\times G/(U_{\lam}\times U_{\lam})\cdot \Del M_{\lam}
$$ and
$$
(G/U_{\lam}\times G/U^-_{\lam})/M_{\lam}=G\times G/(U_{\lam}\times U_{\lam}^-)\cdot \Del M_{\lam}.
$$

We now want to describe $\calC^{\lam}$ explicity. First, consider the embedding $P_{\lam}\hookrightarrow G\times G$
which sends $p$ to $(p,i_{\lam}\circ\pi_{\lam}(p))$. Let $P_{\lam}'$ denote the image of $P_{\lam}$ under this embedding.
Note that $P_{\lam}$ is contained in both $(U_{\lam}\times U_{\lam})\cdot \Del M_{\lam}$ and $(U_{\lam}\times U_{\lam}^-)\cdot \Del M_{\lam}$. Note that $P_{\lam}'$ is contained in both $(U_{\lam}\times U_{\lam})\cdot \Del M_{\lam}$ and in $(U_{\lam}\times U_{\lam}^-)\cdot \Del M_{\lam}$.
Thus we get natural maps
$$
\begin{CD}
(G\times G)/P_{\lam}' @>{\delta}>> G/U_{\lam}\underset{M_{\lam}}\times G/U_{\lam} \\
@V{\gamma}VV  \\
(G/U_{\lam}\underset{M_{\lam}}\times G/U^-_{\lam}.
\end{CD}
$$
\begin{theorem}\label{loc-glob-new}
There is a $G\times G$-equivariant isomorphism $\calC^{\lam}\simeq (G\times G)/P_{\lam}'$ under which $\alpha$ goes to $\gamma$ and $\beta$ goes to $\delta$.
\end{theorem}
\begin{proof}
Since $(G\times G)/P_{\lam}'$ is acted on transitively by $G\times G$ in order to construct the above isomorphism it is enough to check the following things:

(a) The two maps $\calC^{\lam}\to G/P_{\lam}$ given by composing either $\alpha$ or $\beta$ with the natural (first) projection to $G/P_{\lam}$ coincide

(b) The fibers of these maps are $G$-torsors with respect to the action of the 2nd copy of $G$ (corresponding to changing the trivialization of $\calF$ on $\PP^1\backslash \{ 0\}$ by a constant map $\PP^1\backslash \{ 0\}\to G$) are $G$-torsors. Moreover, the fiber over 1 is canonically the trivial torsor (i.e. it is equal to $G$)

(c) Any $p\in P_{\lam}$ acts on the fiber over $1\in G/P_{\lam}$ (which is naturally isomorphic to $G$ by the last sentence of (b)) by the right multiplication by $i_{\lam}\pi_{\lam}(p^{-1})$.

First, we claim that $\calC^{\lam}$ is acted on transitively by $G\times G$. Indeed, for this it is enough to check that a trivialization of $\calO(\lam)$ on $\PP^1\backslash\{ 0\}$ of type $\lam$ is unique up to constant change of trivialization.
This is a special case of the modular description of transversal slices in the to $G(\OO)$-orbits in the affine Grassmannian (cf. the proof of Theorem 2.8 in \cite{BrFi}).

Now (a) follows from the fact that the fiber of the Harder-Narasimhan flag of $\calF$ at $0$ is equal to the Jantzen flag
at $0$ w.r.to the trivialization on $\PP^1\backslash \{0\}$. In view of the above transitivity in order to check this it is enough to assume that $\calF=\calO(\lam)$ with the standard trivialization; in this case it is a straightforward calculation.

Let us check (b). Again, it follows from the proof of Theorem 2.8 in \cite{BrFi} that the space $\overline{\calC}^{\lam}$ consisting just of data (1) and (3) from Subsection \ref{zlam}
is isomorphic to $G/P_{\lam}$, where the action of $G$ is given by changing the trivialization by a constant map.
More precisely, the claim is that $G$ acts transitively on such data and given any $(\calF,\kap)\in \overline{\calC}^{\lam}$
the group of all $g\in \Aut(\calF)$ which commute with $\kap$ maps injectively to the group $\Aut(\calF_0)$ (automorphisms of the fiber of $\calF$ at $0$) and its image consists of those automorphisms of $\calF_0$ which preserve the restriction of the Harder-Narasimhan flag of $\calF$ to $0$.  This implies that if we add to $(\calF,\kap)$ a trivialization of $\calF_0$ such that the restriction of the Harder-Narasimhan flag of $\calF$ to $0$ gives a fixed $P_{\lam}$-structure on the corresponding trivial $G$-torsor, then $G$ still acts transitively but with trivial stabilizer. This is precisely the first assertion of (b). For the 2nd assertion we just need to say that we have the bundle $\calO(\lam)$ with a natural trivilization on $\PP^1\backslash \{ 0\}$ and on $\PP^1\backslash \{ z\}$ which gives rise to a point of $\calC^{\lam}$ which sits over $1\in G/P_{\lam}$, which trivializes the corresponding torsor. Part (c) is then straightforward.
\end{proof}

\begin{corollary}
Let $\kk=\FF_q$. Then $\calC^{\lam}$ (viewed as a correspondence) defines an isomorphism between $\CC[\calH^{\lam}(\kk)]$ and
$\CC[\bpo^{\lam}(\kk)]$.
\end{corollary}

\section{Proof of Theorems \ref{cent-div} and \ref{lift-gln}}


\subsection{Proof of Theorem \ref{cent-div}}
First of all, it is clear that assertion (i) of Theorem \ref{cent-div} follows from assertion (ii).
Also it is enough to prove (ii) in the case $F=\qlb$ (here $\ell$ is prime to the characteristic of $\kk$).
In this case (ii) follows from the (weak form of the) global Langlands correspondence for functional fields proved by L.~Lafforgue.
Namely, the following statement is contained in \cite{Laf}.

Let $\pi=\otimes ' \pi_\nu$ be an irreducible automorphic representation $G$ for the global field $E=\kk(\PP^1)$ (here $\nu$ runs over places of $E$) such that

a) $\pi_\nu$ is unramfied by $\nu\neq 0,\infty$ (i.e. $\pi_\nu^{G(\OO_{E_\nu})}\neq 0$ for $\nu\neq 0,\infty$);

b) $\pi_\nu^{K_1,\nu}\neq 0$ for $\nu=0,\infty$ where $K_{1,\nu}$ denotes the 1st congruence subgroup of $G(E_\nu)$.

\noindent
Let $\pi_1(\PP^1)_{0,\infty}$ denote the fundamental group of $\PP^1\backslash\{ 0,\infty\}$ tamely ramified at $0$ and $\infty$. Then there exists a homomorphism $\rho:\pi_1(\PP^1)_{0,\infty}\to G^{\vee}(\qlb)$ such that the Satake parameter
of $\pi_\nu$ for $\nu\neq 0,\infty$ (which a priori is a semi-simple conjugacy class in $G^{\vee}(\qlb)$) is equal to $\rho(\Fr_\nu)$ (where $\Fr_\nu$ is the Frobenius element at the place $\nu$ which is well defined as an element of $\pi_1(\PP^1)_{0,\infty}$ up to conjugacy).

The group $\Gam=\pi_1(\PP^1)_{0,\infty}$ is isomorphic to the semidirect product of $\widehat{\ZZ}$ (with generator $\Fr$, corresponding to the Frobenius element at the point $1\in \kk^{\x}$) and
the group $\Gam_0=\underset{\leftarrow}\lim \ \kk_i^{\x}$ (this is the corresponding geometric fundamental group), such that $\Fr\gam \Fr^{-1}=\gam^q$ for $\gam\in \Gam_0$. Thus a homomorphism $\Gam\to G^{\vee}(\qlb)$ is given by a pair $(s,u)$ as in Subsection \ref{setup}. The proof then follows from the following two observations:

1) If $x\in \kk_i^{\x}$ and $x$ does not belong to $\kk_j$ for any $j<i$ then $\rho(\Fr_x)$ is conjugate to $s^i u(y)$ where
$y$ as in Theorem \ref{cent-div}.

2) Every irreducible representation of $A$  does appear as $\pi_0$ for some $\pi$ as above (this follows immediately from
Theorem \ref{loc-glob}).

\subsection{Proof of Lemma \ref{lem-cusp}}
Assume that $G$ is semi-simple. Then it is clear that $P_{\lam}\neq G$ for any $\lam\neq 0$.
Hence Lemma \ref{lem-cusp} follows from (\ref{formula-global}).

\subsection{Proof of Theorem \ref{lift-gln}(i)}
For the remainder of this Section we set $G=\PGL(N)$ and we use freely the notation from Subsection \ref{cusp-int}.
The set of connected components of the stack $\Bun_G(\PP^1)$ (and thus of $\bpo$) is naturally identified $\ZZ/N\ZZ$.
Thus the space $V$ acquires a natural $\ZZ/N\ZZ$ grading $V=\oplus_{d\in \ZZ/N\ZZ} V_d$. We claim that $(V_d)_{\cusp}=0$
unless $d=0$ -- this again follows immediately from (\ref{formula-global}). On the other hand, the operator $h_{D,f}$ maps
$V_d$ to $V_{d+\overline{i}}$ where $\overline{i}$ is the image of $i$ in $\ZZ/N\ZZ$, which shows that $h_{D,f}$ acts by $0$
on $V_{\cusp}$ unless $\overline{i}=0$.

\subsection{Proof of Theorem \ref{lift-gln}(ii)}
Here we give a direct geometric proof of Theorem \ref{lift-gln}(ii). In the next subsection we shall give a proof of
Theorem \ref{lift-gln}(iii) which contains Theorem \ref{lift-gln}(ii) as a special case -- but that proof will be indirect (it will use Theorem \ref{cent-div} (and thus it will rely on our knowledge of the global Langlands correspondence) and the formula for the character of Deligne-Lusztig representation from \cite{DL}).

First of all, let us recall the precise definition of the operator $h_{D,f}$ in our case. Let $D$ for the moment be any subset of $\overline{\kk}^{\x}$ invariant under the Galois group of $\overline{\kk}$ over $\kk$.
Let us first work with $G=\GL(N)$ rather than $\PGL(N)$. Then we can consider triples $(\widetilde{\calF_1},\widetilde{\calF_2},\widetilde{\kap})$ where
$\widetilde{\calF_i}$ is a locally free sheaf on $\PP^1$ of rank $N$ trivialized at $0$ and $\infty$ and $\widetilde{\kap}:\widetilde{\calF_1}\to \widetilde{\calF_2}$ is an injective map such that
$\widetilde{\calF_2}/\widetilde{\calF_1}=\calO_D$ (the structure sheaf of $D$). For $\PGL(N)$ we consider the quotient of this stack by
$\Pic(\PP^1)$ acting diagonally. We denote the resulting stack by $\calH_{D,N}$ and we denote its typical point by $(\calF_1,\calF_2,\kap)$ (here each $\calF_i$ is a $\PGL(N)$-bundle on $\PP^1$ trivialized at $0$ and $\infty$ and $\kap$ is an isomorphism between $\calF_1$ and $\calF_2$ away from $D$ (the isomorphism is not required to respect the trivializations at $0$ and $\infty$)). It can be viewed as a correspondence
from $\bpo$ to $\bpo$. We also denote by $\calH_{D,N}^0$ the locus where both $\calF_i$ are trivial. This is a scheme (rather than a stack). It can be viewed as a correspondence from $G$ to $G$.

Then $h_{D,f}$ is the operator induced by the above correspondence multiplied by $(-q^{-1/2})^{(N-1)\deg(D)}$.
We now specialize to the case when $D$ is the Frobenius orbit of some $x\in \kk_i^{\x}$ (which doesn't belong to any $\kk_j$ with $j<i$). In this case $\deg(D)=i$. In view of the previous subsection we only care about the case when $i$ is divisible by $N$. In this case the number $i(N-1)$ is even and hence
$$
(-q^{-1/2})^{(N-1)\deg(D)}=q^{-\frac{iN}{2}}.
$$

To prove Theorem \ref{lift-gln}(ii) we need the following
\begin{lemma}\label{gln-orbit}
Let $D$ be as above. Assume in addition that $i=N$. Then $x$ defines a conjugacy class in $G=\PGL(N)$; we shall denote this class by
 $\Ome_x\subset G$.
Then as a correspondence from $G$ to $G$ the scheme $\calH_{D,N}$ is isomorphic to the subset of $G\times G$ consisting of elements $(g_1,g_2)$ such that $g_1g_2^{-1}\in \Ome_x$.
\end{lemma}
\begin{proof}
First, let us formulate the analogous statement over $\overline{\kk}$ and also replace $\PGL(N)$ by $\GL(N)$. Then we geet the following statement. Let $(a_1,\cdots,a_N)\in \overline{\kk}^{\x}$
be such that $a_j\neq a_j$ for $i\neq j$. Let $\Ome_{a_1,\cdots,a_N}\subset \GL(N,\kk)$ be the orbits of the diagonal
matrix $(-a_1,\cdots,-a_N)$ in the group $\GL(N,\kk)$. Let $Z$ denote the space of injective maps $\phi:\calO^N\to \calO(1)^N$ such that
$$
\calO(1)^N/\phi(\calO^N)=\bigoplus\limits_{i=0}^N \calO_{a_i}.
$$
We can think about $\phi$ as an $N\times N$ matrix with entries in $\kk[t]$ which is invertible away from $a_1,\cdots,a_N$ and such that
$\phi(t)/t$ is regular at $\infty$ and its value at $\infty$ is invertible.
Then $Z$ is isomorphic to the subset of $\GL(N)\times \GL(N)$ consisting of pairs $(g_1,g_2)$ such that $g_1g_2^{-1}\in \Ome_{a_1,\cdots,a_N}$, where the map to $G\times G$ is given by $(\phi(0),\frac{\phi(t)}{t}(\infty))$.
The proof is as follows. The space of subsheaves $\calE\subset \calO(1)^N$ such that $\calO(1)^N/\calE=\bigoplus\limits_{i=0}^N \calO_{a_i}$ is naturally isomorphic to $(\PP^{N-1})^N$. The open subset corresponding to the condition that $\calE$ is trivial is equal to the complement to all diagonals. Thus the group $\GL(N)$ acts transitively on it (the action is by automorphisms of $\calO(1)^N$). Thus it follows that if in addition we choose a trivialization of $\calE$ (which is what we need to get precisely the scheme $Z$) we get a scheme isomorphic to $\GL(N)\times \Ome$ where $\Ome$ is some adjoint orbit in $\GL(N)$. It remains to remark that the diagonal matrix $\phi(t)=(t-a_1,\cdots,t-a_N)$ lies in $Z$ and in this case
$\phi(0)=(-a_1,\cdots,-a_N)$  and $\frac{\phi(t)}{t}(\infty)$ is the identity matrix.

\end{proof}
Let finish the proof of Theorem \ref{lift-gln}(ii). If $\pi$ is an irreducible representation of $G(\kk)$ then the element
$\sum\limits_{g\in \Ome_x} g\in F[G(\kk)]$ acts in $\pi$ by by a scalar $c(\pi,x)$. Lemma \ref{gln-orbit} implies that if $\pi$ is cuspidal  then $\eta_{D,f}(\pi)=q^{-\frac{N(N-1)}{2}}c(\pi,x)$. On the other hand, it is clear (by computing the trace of $\sum_{g\in \Ome_x} g$ in $\pi$ that
$$
|\Ome_x(\kk)|\chi_\pi(x)=\dim(\pi)\cdot c(\pi,x).
$$
But it is easy to see that
$$
|\Ome_x(\kk)|=|\GL(N,\kk)|/|\kk_N^{\x}|=\prod\limits_{j=1}^{N-1} (q^N-q^j)
$$
On the other hand, it follows from \cite{DL} that if $\pi$ is cuspidal and irreducible then
$$
\dim(\pi)=\prod\limits_{j=1}^{N-1} (q^j-1)=q^{-\frac{N(N-1)}{2}}\prod\limits_{j=1}^{N-1} (q^N-q^j)=q^{-\frac{N(N-1)}{2}}|\Ome_x(\kk)|.
$$
Hence
$$
\eta_{D,f}(\pi)=q^{-\frac{N(N-1)}{2}}c(\pi,x)=\eta_{D,f}(\pi)=q^{-\frac{N(N-1)}{2}}\chi_\pi(x)\frac{|\Ome_x(\kk)|}{\dim(\pi)}=
\chi_\pi(x)
$$
which finishes the proof.

\subsection{Proof of Theorem \ref{lift-gln}(iii)}
Assume first that $x\in \kk_N^{\x}$ and let $[\pi]=(-1)^{N-1}R(T,\theta)$ where $T(\kk)=\kk_N^{\x}/\kk^{\x}$ and
$\theta:T(\kk)\to F^{\x}$ is a character such that $\theta^{q^j}\neq \theta$ for all $0<j<N$ (the latter condition is
equivalent to cuspidality of $\pi$)\footnote{The sign $(-1)^{N-1}$ is needed in order to turn $R(T,\theta)$ into actual representation (as opposed to a virtual representation)}. Let us view $\pi$ as a representation of $G(\OO)$ via the evaluation map
$G(\OO)\to G(\kk)$. Then $\ind_{G(\OO)}^{G(\KK)}(\pi)$ (here ind stands for compact induction) is irreducible and thus
$(\ind_{G(\OO)}^{G(\KK)}(\pi))^{K_1}$ is an irreducible (obviously non-zero) representation of $A$. The corresponding
parameters $(s,u)$ are described as follows.

\begin{lemma} Under the above conditions $u: \Gam_0=\underset{\leftarrow}\lim \ \kk_i^{\x}\to G^{\vee}(F)= \SL(N,F)$ factorizes through $\kk_N^{\x}$
and it is equal to the diagonal matrix $(-1)^{N-1}(\theta,\theta^q,\theta^{q^2},\cdots,\theta^{q^{N-1}})$. Also
$s$ is given by the cyclic permutation $(1\to 2\to \cdots\to N\to 1)$.
\end{lemma}
The proof easily follows from the suitable generalization of \cite{BG-Eis} (specifically, the arguments from the proof  Theorem 2.2.8. in {\em loc. cit.}) to the ramified setting. This is relatively straightforward, but we postpone it for another publication.

 So, $s^N=1$ and the RHS of (\ref{formula})
is equal to
$$
(-1)^{N-1}\sum\limits_{i=0}^{N-1}\theta(x)^{q^i}.
$$
This is precisely the value of the character of $(-1)^{N-1}R(T,\theta)$ at $x$.

We now proceed to the case of general $x$; i.e. we assume that $x\in \kk_i^{\x}$ with $i=aN$ and that $x$ does not belong
to any $\kk_j$ with $j<i$. Then again $s^i=1$ and since $\kk_N\subset \kk_i$ the RHS of (\ref{formula}) becomes just
equal to
$$
(-1)^{N-1}\sum\limits_{i=0}^{N-1}\theta(\Norm_{i,N}(x))^{q^i},
$$
which is again the character of the above representation $\pi$.

\end{document}